\documentclass{amsart}
\usepackage{amssymb}
\usepackage{amscd}
\usepackage{amssymb}
\usepackage{amsthm}

\input xy
\xyoption{all}

\newcommand{\pr}{\mathrm{pr}}

\newcommand{\id}{\mathrm{id}}

\newcommand{\C}{\mathcal C}

\newcommand{\U}{\mathcal U}

\newcommand{\N}{\mathbb N}
\newcommand{\R}{\mathbb R}

\newcommand{\ext}{\mathrm{ext}}
\newcommand{\Cl}{\mathrm{Cl}}

\newtheorem{theorem}{Theorem}[section]
\newtheorem{corollary}[theorem]{Corollary}

\newtheorem{lemma}[theorem]{Lemma}

\newtheorem*{question1}{Question 7.2}
\newtheorem*{question2}{Question 7.3}
\newtheorem{proposition}[theorem]{Proposition}
\theoremstyle{definition}
\newtheorem{df}[theorem]{Definition}

\title{On the openness of the idempotent barycenter map}

\author[T.~Radul]{Taras Radul}
\address[T.~Radul]{Kazimierz Wielki University, Bydgoszcz (Poland) and Ivan Franko National University of Lviv (Ukraine)}
\email{tarasradul@yahoo.co.uk}


\begin{document}
\begin{abstract} We show that the openness of the idempotent barycenter map is equivalent to the openness of the map of max-plus convex combination. As corollary we obtain that the idempotent barycenter map is open for the spaces of idempotent measures.
\end{abstract}

\maketitle

\section{Introduction}

The notion of idempotent (Maslov) measure finds important applications in different
parts of mathematics, mathematical physics and economics (see the survey article
\cite{Litv} and the bibliography therein). Topological and categorical properties of the functor of idempotent measures were studied in \cite{Zar}. There are some parallels between the theory of probability measures and idempotent measures (see for example \cite{Radul}).

The problem of the openness of the barycentre map of probability measures was investigated  in \cite{Fed}, \cite{Fed1},  \cite{Eif}, \cite{OBr} and \cite{Pap}. In particular, it is proved in \cite{OBr} that the barycentre map for a compact convex set  in a locally convex space is open iff the map $(x, y)\mapsto 1/2 (x + y)$ is open.

Zarichnyj defined in \cite{Zar} the idempotent barycentre map for idempotent measures and asked the following two questions:

\begin{question1}\cite{Zar} Characterize the class of max-plus convex compact spaces for which the idempotent barycenter map is open. In particular, is the latter property equivalent to the openness of the map $(x, y)\mapsto x \oplus y$?
\end{question1}

It is proved in \cite{Fed1} that the product of barycentrically open compact convex sets (i.e. compact convex
sets for which the barycentre map  is open) is
again barycentrically open.

\begin{question2}\cite{Zar}Is an analogous fact true for idempotent barycentrically open max-plus
convex sets?
\end{question2}

In this paper we characterize when  the idempotent barycenter map is open. However we show that the openness of the idempotent barycenter map is not equivalent to the openness of the map $(x, y)\mapsto x \oplus y$. We also  answer in the negative  the second question.

\section{Idempotent measures: preliminaries}

In the sequel, all maps will be assumed to be continuous. Let $X$ be a compact Hausdorff space. We shall denote by $C(X)$ the
Banach space of continuous functions on $X$ endowed with the sup-norm. For any $c\in\ R$ we shall denote  by $c_X$ the
constant function on $X$ taking the value $c$.

Let $\R_{\max}=\R\cup\{-\infty\}$ be the metric space endowed with the metric $\varrho$ defined by $\varrho(x, y) = |e^x-e^y|$.
 Following the notation of idempotent mathematics (see e.g., \cite{MS}) we use the
notations $\oplus$ and $\odot$ in $\R$ as alternatives for $\max$ and $+$ respectively. The convention $-\infty\odot x=-\infty$ allows us to extend $\odot$ and $\oplus$  over $\R_{\max}$.

Max-plus convex sets were introduced in \cite{Z}.
Let $\tau$ be a cardinal number. Given $x, y \in \R^\tau$ and $\lambda\in\R_{\max}$, we denote by $y\oplus x$ the coordinatewise
maximum of x and y and by $\lambda\odot x$ the vector obtained from $x$ by adding $\lambda$ to each of its coordinates. A subset $A$ in $\R^\tau$ is said to be  max-plus convex if $\alpha\odot a\oplus  b\in A$ for all $a, b\in A$ and $\alpha\in\R_{\max}$ with $\alpha\le 0$. It is easy to check that $A$  is   max-plus convex iff $\oplus_{i=1}^n\lambda_i\odot x_i\in A$ for all $x_1,\dots, x_n\in A$ and $\lambda_1,\dots,\lambda_n\in\R_{\max}$ such that $\oplus_{i=1}^n\lambda_i=0$. In the following by max-plus convex compactum we mean a max-plus convex compact subset of $\R^\tau$.

We denote by $\odot:\R\times C(X)\to C(X)$ the map acting by $(\lambda,\varphi)\mapsto \lambda_X+\varphi$, and by $\oplus:C(X)\times C(X)\to C(X)$ the map acting by $(\psi,\varphi)\mapsto \max\{\psi,\varphi\}$.

\begin{df}\cite{Zar} A functional $\mu: C(X) \to \R$ is called an idempotent  measure (a Maslov measure) if

\begin{enumerate}
\item $\mu(1_X)=1$;
\item $\mu(\lambda\odot\varphi)=\lambda\odot\mu(\varphi)$ for each $\lambda\in\R$ and $\varphi\in C(X)$;
\item $\mu(\psi\oplus\varphi)=\mu(\psi)\oplus\mu(\varphi)$ for each $\psi$, $\varphi\in C(X)$.
\end{enumerate}

\end{df}

Let $IX$ denote the set of all idempotent  measures on a compactum $X$. We consider
$IX$ as a subspace of $\R^{C(X)}$. It is shown in \cite{Zar} that $IX$ is a compact max-plus subset of $\R^{C(X)}$. The construction $I$ is  functorial what means that for each continuous map $f:X\to Y$ we can consider a continuous map $If:IX\to IY$ defined as follows $If(\mu)(\psi)=\mu(\psi\circ f)$ for $\mu\in IX$ and $\psi\in C(Y)$.

By $\delta_{x}$ we denote the Dirac measure supported by the point $x\in X$. We can consider a map $\delta X:X\to IX$ defined as $\delta X(x)=\delta_{x}$, $x\in X$. The map $\delta X$ is continuous, moreover it is an embedding \cite{Zar}. It is also shown in \cite{Zar} that the set $$I_\omega X=\{\oplus_{i=1}^n\lambda_i\odot\delta_{x_i}\mid\lambda_i\in\R_{\max},\ i\in\{1,\dots,n\},\ \oplus_{i=1}^n\lambda_i=0,\ x_i\in X,\ n\in\N\},$$ (i.e., the set of idempotent probability measures of finite support) is dense in $IX$. Let us also remark that for a finite compactum $X=\{1,\dots,n\}$ we have $IX=\{\oplus_{i=1}^n\lambda_i\odot\delta_i\mid\lambda_i\in\R_{\max},$ such that $\oplus_{i=1}^n\lambda_i=0\}$ \cite{Zar}.

Let $A\subset  \R^T$ be a compact max-plus convex subset. For each $t\in T$ we put $f_t=\pr_t|_A:A\to \R$ where $\pr_t:\R^T\to\R$ is the natural projection.    Given $\mu\in IA$, the point $\beta_A(\mu)\in\R^T$ is defined by the conditions $\pr_t(\beta_A(\mu))=\mu(f_t)$ for each $t\in T$. It is shown in \cite{Zar} that $\beta_A(\mu)\in  A$ for each $\mu\in I(A)$ and the map $\beta_A : I(A)\to A$ is continuous.
The map $\beta_A$ is called the idempotent barycenter map. It follows from results of \cite{Zar} that for each compactum $X$ we have $\beta_{IX}\circ I(\delta X)=\id_{IX}$ and for each map $f:X\to Y$ between compacta $X$ and $Y$ we have $\beta_{IY}\circ I^2f=If\circ\beta_{IX}$.

\section{The openness of max-plus convex combination of idempotent measures}
Put $J=\{(t,p)\in [-\infty,0]\times [-\infty,0]\mid t\oplus p=0\}$.
Let $X$ be a  max-plus convex compactum.  We consider a map $s_X:X\times X\times J\to X$ defined by the formula $s_X(x,y,t,p)=t\odot x\oplus p\odot y$.

Since the set $I_\omega X$  is dense in $IX$,  the following lemma can be obtained by direct checking for  idempotent  measures of finite support.

\begin{lemma}\label{Comm} Let $f:X\to Y$ be a continuous  map between compacta $X$ and $Y$. The  diagram
$$ \CD
I(X)\times I(X)\times J  @>I(f)\times I(f)\times\id_{J}>>   I(Y)\times I(Y)\times J \\
@VV s_{I(X)} V      @VV s_{I(Y)} V\\
I(X)  @>I(f)>>  I(Y) \\
\endCD
$$ is commutative.
\end{lemma}

 The main goal of this section is to prove that the map $s_{IX}:IX\times IX\times J\to IX$ is open for each compactum $X$. We start with a finite $X$.

\begin{lemma}\label{Fin} The map $s_{IX}$ is open for each finite compactum $X$.
\end{lemma}

\begin{proof} Let $X=\{1,\dots,n\}$. Since the functor $I$ preserves the weight \cite{Zar}, the compactum $IX$ is metrizable.  Consider any $(\lambda,\beta,t,p)\in IX\times IX\times J$ and a sequence $(\alpha^j)$ in $IX$ converging to $t\odot \lambda\oplus p\odot \beta$. It is enough to find sequences $(\lambda^j)$, $(\beta^j)$ in $IX$ and a sequence $(t^j,p^j)$ in $J$ such that the sequence $(\lambda^j,\beta^j,t^j,p^j)$ converges to $(\lambda,\beta,t,p)$ and $t^j\odot \lambda^j\oplus p^j\odot\beta^j=\alpha^j$ for each $j\in \N$.

We have $\lambda=\oplus_{i=1}^n\lambda_i\odot\delta_i$, $\beta=\oplus_{i=1}^n\beta_i\odot\delta_i$ and $\alpha^j=\oplus_{i=1}^n\alpha^j_i\odot\delta_i$ where $\lambda_i$, $\beta_i$, $\alpha^j_i\in\R_{\max}$ such that $\oplus_{i=1}^n\lambda_i=\oplus_{i=1}^n\beta_i=\oplus_{i=1}^n\alpha^j_i=0$. Then $t\odot \lambda\oplus p\odot \beta=\oplus_{i=1}^n(t\odot \lambda_i\oplus p\odot\beta_i)\odot\delta_i$ and we have that the sequence $\alpha^j_i$ converges to $t\odot \lambda_i\oplus p\odot\beta_i$ for each $i\in \{1,\dots,n\}$. We can assume (passing to a subsequence if necessary) that there exists $i_0\in\{1,\dots,n\}$ such that $\alpha^j_{i_0}=0$ for each $j\in\N$.

Consider the case $t=p$. Then we have $t=p=0$.  We can represent $X=A\sqcup B\sqcup C$ where $A=\{i\in\{1,\dots,n\}|\lambda_i<\beta_i\}$, $B=\{i\in\{1,\dots,n\}|\lambda_i>\beta_i\}$ and $C=\{i\in\{1,\dots,n\}|\lambda_i=\beta_i\}$. We can assume that $\alpha^j_i>\frac{\lambda_i+\beta_i}{2}$ for each $j\in A\cup B$.

Consider the subcase $i_0\in C$. Then $\lambda_{i_0}=\beta_{i_0}=0$. Put
$$
\lambda^j_i=\begin{cases}
\lambda_i,&i\in A,\\
\alpha^j_i,&i\notin A\end{cases}
$$

and

$$
\beta^j_i=\begin{cases}
\beta_i,&i\in B,\\
\alpha^j_i,&i\notin B\end{cases}
$$

We have $\lambda^j_i\le 0$, $\beta^j_i\le 0$ and $\lambda^j_{i_0}=\beta^j_{i_0}=0$. Put $\lambda^j=\oplus_{i=1}^n\lambda^j_i\odot\delta_i$ and $\beta^j=\oplus_{i=1}^n\beta^j_i\odot\delta_i$. Then the sequence $(\lambda^j,\beta^j,0,0)$ converges to $(\lambda,\beta,0,0)$ and $\lambda^j\oplus \beta^j=\alpha^j$ for each $j\in \N$.

Consider the subcase $i_0\in A$. (The proof is analogous for the subcase $i_0\in B$.)  Put $c^j=\max\{\alpha^j_i|i\notin A\}$. The sequence $(c^j)$ converges to $0$.

Put
$$
\lambda^j_i=\begin{cases}
\lambda_i,&i\in A,\\
\alpha^j_i-c^j,&i\notin A\end{cases}
$$

and

$$
\beta^j_i=\begin{cases}
\beta_i,&i\in B,\\
\alpha^j_i,&i\notin B\end{cases}
$$

We have $\lambda^j_i\le 0$, $\beta^j_i\le 0$ and $\beta^j_{i_0}=0$. We also have  $\lambda^j_{i}=0$ for each $i\notin A$ such that $c^j=\alpha^j_i$. Put $\lambda^j=\oplus_{i=1}^n\lambda^j_i\odot\delta_i$ and $\beta^j=\oplus_{i=1}^n\beta^j_i\odot\delta_i$. Then the sequence $(\lambda^j,\beta^j,c^j,0)$ converges to $(\lambda,\beta,0,0)$ and $c^j\odot \lambda^j\oplus \beta^j=\alpha^j$ for each $j\in \N$.

Finally consider the case $t<p$. (The proof is analogous for the case $p<t$.) We have $p=0$. If $t=-\infty$, the sequence $\alpha^j$ converges to $\beta$. We have that the sequence $(\lambda,\alpha^j,-\infty,0)$ converges to $(\lambda,\beta,-\infty,0)$ and $-\infty\odot\lambda\oplus\alpha^j=\alpha^j$.

Now, consider $t>-\infty$.  We have $X=A\sqcup B\sqcup C$ where $A=\{i\in\{1,\dots,n\}|t\odot\lambda_i<\beta_i\}$, $B=\{i\in\{1,\dots,n\}|t\odot\lambda_i>\beta_i\}$ and $C=\{i\in\{1,\dots,n\}|t\odot\lambda_i=\beta_i\}$. We can assume that $\alpha^j_i>\frac{t+\lambda_i+\beta_i}{2}$ for each $j\in A\cup B$.

We also have  $i_0\in A$ and $\beta_{i_0}=0$. Consider $D=\{i\in X\setminus A|\lambda_i>-\infty\}$. Put $c^j=0$ if $D=\emptyset$. For $D\neq\emptyset$  put $c^j=\max\{\alpha^j_i-t-\lambda_i|i\in D\}$ if there exists $s\in A$ such that $\lambda_s=0$ and $c^j=\max\{\alpha^j_i-t|i\in D\}$ otherwise. The sequence $(c_j)$ converges to $0$.

Put
$$
\lambda^j_i=\begin{cases}
\lambda_i,&i\in A,\\
\alpha^j_i-c^j-t,&i\notin A\end{cases}
$$

and

$$
\beta^j_i=\begin{cases}
\beta_i,&i\in B,\\
\alpha^j_i,&i\notin B\end{cases}
$$

We have $\lambda^j_i\le 0$, $\beta^j_i\le 0$ and $\beta^j_{i_0}=0$. If $\lambda^j_{s}\ne 0$ for each $s\in A$, we  have  $\lambda^j_{i}=0$ for each $i\notin A$ such that $c^j=\alpha^j_i-t$. Put $\lambda^j=\oplus_{i=1}^n\lambda^j_i\odot\delta_i$ and $\beta^j=\oplus_{i=1}^n\beta^j_i\odot\delta_i$. Then the sequence $(\lambda^j,\beta^j,t\odot c^j,0)$ converges to $(\lambda,\beta,t,0)$ and $(c^j\odot t)\odot \lambda^j\oplus \beta^j=\alpha^j$ for each $j\in \N$.
\end{proof}

Let $$ \CD
X_1  @>p>>   X_2 \\
@VVf_1V      @VVf_2V  \\
Y_1  @>q>>   Y_2
\endCD
$$

be a commutative diagram. The map $\chi:X_1\to
X_2\times{}_{Y_2}Y_1= \{(x,y)\in X_2\times Y_1\mid f_2(x)=q(y)\}$
defined by $\chi(x)=(p(x),f_1(x))$ is called a characteristic map
of this diagram. The diagram is called bicommutative  if the map $\chi$ is onto.

\begin{lemma}\label{0dim} The map $s_{IX}$ is open for each 0-dimensional compactum $X$.
\end{lemma}

\begin{proof} Represent $X$ as
the limit of an inverse system $\C=\{X_\alpha,p_\beta^\alpha,A\}$
consisting of finite compacta and epimorphisms. It is easy to check that  $s_{IX}=\lim\{s_{I(X_\alpha)}\}$.  By Proposition 2.10.9
\cite{TZ} and Lemma \ref{Fin} in order to prove that the map $s_{IX}$ is open, it is
sufficient to prove that the diagram
$$ \CD
I(X_\alpha)\times I(X_\alpha)\times J  @>I(p_\beta^\alpha)\times I(p_\beta^\alpha)\times\id_{J}>>   I(X_\beta)\times I(X_\beta)\times J \\
@VV s_{I(X_\alpha)} V      @VV s_{I(X_\beta)} V\\
I(X_\alpha)  @>I(p_\beta^\alpha)>>  I(X_\beta) \\
\endCD
$$
(which is commutative by Lemma \ref{Comm}) is bicommutative for each $\alpha\ge\beta$.

Without loss of generality, one may assume that $$X_\alpha=\{x_1,\dots,x_{n+1}\},\  X_\beta=\{y_1,\dots,y_{n}\}$$ (all the points are assumed to be distinct) and the map $p_\beta^\alpha$ acts as follows: $p_\beta^\alpha(x_i)=y_i$ for each $i\in\{1,\dots,n\}$ and $p_\beta^\alpha(x_{n+1})=y_n$. Thus, given $(\nu,(\mu, \alpha,t,q))\in I(X_\alpha)\times{}_{I(X_\beta)}I(X_\beta)\times I(X_\beta)\times J$ one can write $\nu=\oplus_{i=1}^{n+1}\nu_i\odot\delta_{x_i}$, $\mu=\oplus_{i=1}^{n}\mu_i\odot\delta_{y_i}$ and $\alpha=\oplus_{i=1}^{n}\alpha_i\odot\delta_{y_i}$.

Consider the case $q=0$, the proof is analogous for the case $t=0$.

Since $I(p_\beta^\alpha)(\nu)=t\odot\mu\oplus\alpha$, we have
$$\nu_i=t\odot\mu_i\oplus\alpha_i,\ i\in\{1,\dots,n-1\}$$
and
$$\nu_n\oplus\nu_{n+1}=t\odot\mu_n\oplus\alpha_n.$$

Put $$\lambda_i=\mu_i,\ \eta_i=\alpha_i,\ i\in\{1,\dots,n-1\},$$
$$\lambda_n=\min\{\mu_n,\nu_n-t\},\ \lambda_{n+1}=\min\{\mu_n,\nu_{n+1}-t\}$$
and
$$\eta_n=\min\{\alpha_n,\nu_n\},\ \eta_{n+1}=\min\{\alpha_n,\nu_{n+1}\}.$$
It is a routine checking  that
$$\lambda_n\oplus\lambda_{n+1}=\mu_n,\ \eta_n\oplus\eta_{n+1}=\alpha_n$$
and
$$\nu_n=t\odot\lambda_n\oplus\eta_n,\ \nu_{n+1}=t\odot\lambda_{n+1}\oplus\eta_{n+1}.$$

Hence we obtain $s_{I(X_\alpha)}(\lambda,\eta,t,0)=\nu$ and $I(p_\beta^\alpha)\times I(p_\beta^\alpha)\times\id_{J}(\lambda,\eta,t,0)=(\mu, \alpha,t,0)$ for $\lambda=\oplus_{i=1}^{n+1}\lambda_i\odot\delta_{x_i}$ and $\eta=\oplus_{i=1}^{n+1}\eta_i\odot\delta_{x_i}$.
\end{proof}

\begin{theorem}\label{Gen} The map $s_{IX}$ is open for each  compactum $X$.
\end{theorem}

\begin{proof} Choose a continuous onto map $f:Y\to X$ such that $Y$ is a 0-dimensional compactum and there exists a continuous $l:X\to IY$ such that $If\circ l=\delta X$. Existence of such map was proved in \cite{Zar}. (It is called an idempotent Milyutin map.)

Define a map $\gamma:IX\to IY$ by the formula $\gamma=\beta_{IY}\circ Il$. Then we have $If\circ \gamma=If\circ\beta_{IY}\circ Il=\beta_{IX}\circ I^2f\circ Il=\beta_{IX}\circ I(If\circ l)=\beta_{IX}\circ I(\delta X)=\id_{IX}$. Since $I$ preserves surjective maps, $\gamma$ is an embedding and we can consider $IX$ as a subset of $IY$. (We identify $IX$ with $\gamma(IX)$).

Put $T=s_{IY}^{-1}(IX)$. The map $s_{IY}|_T:T\to IX$ is open. The equality $s_{IY}|_T=s_{IX}\circ(If\times If\times\id_{J}|_T)$ follows from Lemma \ref{Comm} and the equality $If\circ \gamma=\id_{IX}$. Hence $s_{IX}$ is open being a left divisor of the open map $s_{IY}|_T$.
\end{proof}

\section{The main result}

We  characterize openness of the barycenter map in this section. Since the set $I_\omega X$  is dense in $IX$,  the following lemma can be obtained by direct checking for  idempotent  measures of finite support.

\begin{lemma}\label{Com} The equality $\beta_X\circ s_{IX}=s_X\circ(\beta_X\times\beta_X\times\id_J)$ holds for each max-plus convex compactum $X$.
\end{lemma}

\begin{corollary}\label{afin} Let $X$ be a max-plus convex compactum, $\mu_1,\dots,\mu_k\in IX$ and $\lambda_1,\dots,\lambda_k\in [-\infty,0]$ be numbers such that $\max\{\lambda_1,\dots,\lambda_k\}=0$.  Then we have $\beta_X(\oplus_{i=1}^k\lambda_i\odot\mu_i)=\oplus_{i=1}^k\lambda_i\odot\beta_X(\mu_i)$.
\end{corollary}

The notion of density for an idempotent measure was introduced in \cite{A}. Let $\mu\in IX$. Then we can define a function $d_\mu:X\to [-\infty,0]$ by the formula $d_\mu(x)=\inf\{\mu(\varphi)|\varphi\in C(X)$ such that $\varphi\le 0$ and $\varphi(x)=0\}$, $x\in X$. The function $d_\mu$ is upper semicontinuous and is called the density of $\mu$. Conversely, each upper semicontinuous function $f:X\to [-\infty,0]$ with $\max f = 0$ determines an idempotent measure $\nu_f$
by the formula $\nu_f(\varphi) = \max\{f(x)\odot\varphi(x) | x \in X\}$, for $\varphi\in C(X)$.

\begin{lemma}\label{dens} Let $X$ be a max-plus convex compactum, $\mu\in IX$ and $U$ be an open neighborhood of $\mu$. Then there exists $\nu\in I_\omega X\cap U$ such that $\beta_X(\nu)=\beta_X(\mu)$.
\end{lemma}

\begin{proof} By $d_\mu$ we denote the density of  $\mu$. Let $\U=\{U_1,\dots,U_k\}$ be a closed max-plus convex cover of $X$. For $i\in\{1,\dots,k\}$ put $s_i=\max\{d_\mu(y)\mid y\in U_i\}$ and define a function $d_i:X\to [-\infty,0]$ by the formula $$d_i(x)=\begin{cases}
d_\mu(x)-s_i,&x\in U_i,\\
-\infty,&x\notin U_i.\end{cases}$$
It is easy to check that $d_i$ is an upper semicontinuous function  with $\max d_i = 0$. Denote by $\mu_i$ the idempotent measure determined  by $d_i$ and put $x_i=\beta_X(\mu_i)$.

Define $\nu_{\U}\in I_\omega X$ by the formula $\nu_{\U}=\oplus_{i=1}^ks_i\odot\delta_{x_i}$. By Corollary \ref{afin} we have $\beta_X(\nu_{\U})=\beta_X(\oplus_{i=1}^ks_i\odot\mu_i)$. Since $\oplus_{i=1}^ks_i\odot\mu_i=\mu$, we obtain $\beta_X(\nu_{\U})=\beta_X(\mu)$.

Now $\{\nu_{\U}\}$ forms a net where the set of all finite closed max-plus convex covers is ordered by refinement. Then $\nu_{\U}\rightarrow\mu$.
\end{proof}

\begin{theorem}\label{main} Let $X$ be a max-plus convex compactum. Then the following statements are equivalent:
\begin{enumerate}
\item the map $\beta_X|_{I_\omega X}:I_\omega X\to X$ is open;
\item the map $\beta_X$ is open;
\item the map $s_X$ is open.
\end{enumerate}
\end{theorem}

\begin{proof} The implication 1.$\Rightarrow$ 2. follows from Lemma \ref{dens}.

2.$\Rightarrow$ 3. Consider any $(x,y,t)\in X\times X\times J$ and let $W$ be an open neighborhood of $(x,y,t)$. We can suppose that $W=V\times U\times O$ where $V$, $U$ and $O$ are open neighborhoods of $x$, $y$ and $t$ in $X$, $X$ and $J$ correspondingly. Since the map $s_{IX}$ is open by Theorem \ref{Gen}, the set $s_{IX}(\beta_X^{-1}(V)\times\beta_X^{-1}(U)\times O)$ is open in $IX$. Then $\beta_X\circ s_{IX}(\beta_X^{-1}(V)\times\beta_X^{-1}(U)\times O)$ is open in $X$.

Let us show that $\beta_X\circ s_{IX}(\beta_X^{-1}(V)\times\beta_X^{-1}(U)\times O)=s_X(V\times U\times O)$. Consider any $y\in\beta_X\circ s_{IX}(\beta_X^{-1}(V)\times\beta_X^{-1}(U)\times O)$. Then there exists $(\mu,\nu,p)\in \beta_X^{-1}(V)\times\beta_X^{-1}(U)\times O$ such that $\beta_X\circ s_{IX}(\mu,\nu,p)=y$. It follows from Lemma \ref{Com} that $\beta_X\circ s_{IX}(\mu,\nu,p)=s_X(\beta_X(\mu),\beta_X(\nu),p)$, hence $y\in s_X(V\times U\times O)$.

Now take any $z\in s_X(V\times U\times O)$. Then there exists  $(r,q,p)\in V\times U\times O$ such that $z=s_X(r,q,p)$. By Lemma \ref{Com} we have $z=\beta_X\circ s_{IX}(\delta_r,\delta_q,p)$. Hence $z\in \beta_X\circ s_{IX}(\beta_X^{-1}(V)\times\beta_X^{-1}(U)\times O)$.

3.$\Rightarrow$ 1. Consider any  $\nu=\oplus_{i=1}^k\lambda_i\odot\delta_{x_i}\in I_\omega X$. We will prove that for each net  $\{x^\alpha\}$ converging to $\beta_X(\nu)$ there exists a net $\{\nu^\alpha\}$ converging to $\nu$ such that $\beta_X(\nu^\alpha)=x^\alpha$ for each $\alpha$.

We use the induction by $k$. For $k=1$ the statement is obvious. Let us assume that we have proved the statement for each $k\le l\ge 1$.

Consider $k=l+1$. Then  $\nu=\oplus_{i=1}^{l+1}\lambda_i\odot\delta_{x_i}$. We can assume that there exists $i\in \{1,\dots,l\}$ such that $\lambda_i=0$. Put  $\nu_1=\oplus_{i=1}^l\lambda_i\odot\delta_{x_i}$. We have $\nu_1\oplus\lambda_{l+1}\odot \delta_{x_{l+1}}=\nu$. Hence $\beta_X(\nu_1)\oplus\lambda_{l+1}\odot x_{l+1}=\beta_X(\nu)$ by Corollary \ref{afin}.

Consider any net $\{x^\alpha\}$ in $X$ converging to $\beta_X(\nu)$. Since the map $s_X$ is open, there exists a net $\{(y^\alpha,x_{l+1}^\alpha,t^\alpha,\lambda^\alpha_{l+1})\}$ in $X\times X\times J$ converging to $(\beta_X(\nu_1),x_{l+1},0,\lambda_{l+1})$ such that $t^\alpha\odot y^\alpha\oplus\lambda^\alpha_{l+1}\odot x^\alpha_{l+1}=x^\alpha$. By the induction assumption there exists a net $\{\nu_1^\alpha\}$ converging to $\nu_1$ such that $\beta_X(\nu_1^\alpha)=y^\alpha$. Then the net $\{t^\alpha\odot \nu_1^\alpha\oplus \lambda^\alpha_{l+1}\odot\delta_{x^\alpha_{l+1}}\}$ converges to $\nu$ and $\beta_X(t^\alpha\odot \nu_1^\alpha\oplus \lambda^\alpha_{l+1}\odot\delta_{x^\alpha_{l+1}})=x^\alpha$ for each $\alpha$.
\end{proof}

Theorems \ref{Gen} and \ref{main} yield the following corollary.

\begin{corollary}\label{free} The map $\beta_{IX}$ is open for each compactum $X$.
\end{corollary}

Let us consider an example of a max-plus convex compactum $K$ such that the map $\beta_K$ is open but the map $(x, y)\mapsto x \oplus y$ is not. This example gives a negative answer to the second part of Zarichnyi Question 7.2 and demonstrate some difference  between the theory of probability measures and idempotent measures. Put $K=ID$ where $D=\{0,1\}$ is a two-point discrete compactum. Then the map $\beta_K$ is open by Corollary \ref{free}. Put $\nu_t=t\odot\delta_0\oplus\delta_1$. Then the sequence $\{\nu_{-1/i}\}$ converges to $\nu_0=\delta_0\oplus\delta_1$. Consider a function $\varphi\in C(D)$ defined by the formula $\varphi(i)=i$, $i\in\{0,1\}$ and an open neighborhood $O=\{(\mu,\gamma)\in ID\times ID\mid |\mu(\varphi)|<1/2\}$ of $(\delta_0,\delta_1)$ in $ID\times ID$.  Consider any pair $(\alpha,\beta)\in ID\times ID$ such that $\alpha\oplus\beta=\nu_{-1/i}$ for some $i\in\N$. We have $\alpha=\alpha_0\odot\delta_0\oplus\alpha_1\odot\delta_1$ and $\beta=\beta_0\odot\delta_0\oplus\beta_1\odot\delta_1$ for some $\alpha_0$, $\alpha_1$, $\beta_0$, $\beta_1\in\R_{\max}$ such that $\alpha_0\oplus\alpha_1=\beta_0\oplus\beta_1=0$. Since $\alpha_0\le-1/i$, we have
$\alpha_1=0$ and $\alpha(\varphi)=1\ge1/2$. Hence $(\alpha,\beta) \notin O$ and the map $(x, y)\mapsto x \oplus y$ is not open.

\section{$I$-barycentrically open compacta and extremal points}

A max-plus convex compactum $K$ such that the map $\beta_K$ is open is called $I$-{\it barycentrically open compactum}. Corollary \ref{free} states, in fact, that the class of $I$-barycentrically open compacta contains
all compacta $IX$. We try to find more $I$-barycentrically open compacta in this section.

Let $\{X_\alpha\}_{\alpha\in A}$ be a family of  max-plus convex compacta. Then the product $X=\prod_{\alpha\in A}X_\alpha$ has a natural structure of max-plus convexity with coordinatewise  operation: $t\odot(x_\alpha)\oplus(y_\alpha)=(t\odot x_\alpha\oplus y_\alpha)$ where $(x_\alpha)$, $(y_\alpha)\in X$ and $t\in [-\infty,0]$.

Denote $Q=[a,b]^H$ where $H$ is any set (finite or infinite) and $a$, $b\in\R$ such that $a\le b$.

\begin{theorem}\label{cube} The cube $Q$ is $I$-barycentrically open.
\end{theorem}

\begin{proof} By Theorem \ref{main} it is enough to prove that the map $s_Q$ is open. Let $t=\alpha\odot x\oplus\beta\odot y$ where $(\alpha,\beta)\in J$ and $x$, $y\in Q$. Consider any net $(t_i)$ in $Q$ converging to $t$. We need to find a net $(x_i,y_i,\alpha_i,\beta_i)$ in $Q\times Q\times J$ converging to $(x,y,\alpha,\beta)$ such that $t_i=\alpha_i\odot x_i\oplus\beta_i\odot y_i$. We have $x=(x^\gamma)_{\gamma\in H}$, $y=(y^\gamma)_{\gamma\in H}$, $t=(t^\gamma)_{\gamma\in H}$ and $t_i=(t_i^\gamma)_{\gamma\in H}$ for each $i$. It is enough to find for each $\gamma\in H$ a net $(x^\gamma_i,y^\gamma_i,\alpha_i,\beta_i)$ in $[a,b]\times [a,b]\times J$ converging to $(x^\gamma,y^\gamma,\alpha,\beta)$ such that $t^\gamma_i=\alpha_i\odot x^\gamma_i\oplus\beta_i\odot y^\gamma_i$ ($\alpha_i$ and $\beta_i$ do not depend on $\gamma$!).

Since $(\alpha,\beta)\in J$, we have $\alpha=0$ or  $\beta=0$. We assume $\beta=0$. (The proof is analogous for $\alpha=0$). Put $\beta_i=\beta=0$.

If $\alpha=-\infty$, we put $\alpha_i=\alpha=-\infty$, $y^\gamma_i=t^\gamma_i$ and  $x^\gamma_i=x^\gamma$.

Consider the case $\alpha=0$.  If $y^\gamma<x^\gamma=t^\gamma$ we put $x^\gamma_i=t^\gamma_i$ and
$$
y^\gamma_i=\begin{cases}
y^\gamma,&y^\gamma<t^\gamma_i,\\
t^\gamma_i,&y^\gamma\ge t^\gamma_i.\end{cases}
$$

If $x^\gamma<y^\gamma=t^\gamma$, we put $y^\gamma_i=t^\gamma_i$ and
$$
x^\gamma_i=\begin{cases}
x^\gamma,&x^\gamma<t^\gamma_i,\\
t^\gamma_i,&x^\gamma\ge t^\gamma_i.\end{cases}
$$

Finally, if $x^\gamma=y^\gamma=t^\gamma$, we put $y^\gamma_i=x^\gamma_i=t^\gamma_i$. It is easy to see that $(x^\gamma_i,y^\gamma_i,0,0)$ is a net we are looking for.

Let $\alpha<0$. We consider the case when the set $H$ is finite. Then we can assume that the  net $(t_i)$ is a sequence ($i\in\N$). Since $(t_i)$ converges to $t$, we can assume that $t^\gamma_i-t^\gamma\le-\alpha$ for each $i\in\N$ and $\gamma\in H$. Then put  $\alpha_i=\alpha\odot\max\{t^\beta_i-t^\beta\mid\beta\in H\}$ for $i\in\N$. We have that the sequence  $(\alpha_i)$ converges to $\alpha$ and $\alpha_i\le 0$.

If $y^\gamma<\alpha\odot x^\gamma=t^\gamma$ we put $x^\gamma_i=-\alpha_i\odot t^\gamma_i$ and
$$
y^\gamma_i=\begin{cases}
y^\gamma,&y^\gamma<t^\gamma_i,\\
t^\gamma_i,&y^\gamma\ge t^\gamma_i.\end{cases}
$$

 We have $x^\gamma_i=-\alpha_i\odot t^\gamma_i\rightarrow-\alpha\odot t^\gamma=x^\gamma$ and $x^\gamma_i=-\alpha_i\odot t^\gamma_i=-\alpha-\max\{t^\beta_i-t^\beta\mid\beta\in H\}\odot t^\gamma_i\le-\alpha\odot t^\gamma=x^\gamma\le b$. Evidently $x^\gamma_i\ge a$. Since  $y^\gamma<\alpha\odot x^\gamma=t^\gamma$, there exists $i_0\in\N$ such that $y^\gamma<t^\gamma_i$ for each $i\ge i_0$. Hence $y^\gamma_i\rightarrow y^\gamma$. Finally, we have $y^\gamma_i\le t^\gamma_i$ and $\alpha_i\odot x^\gamma_i=t^\gamma_i$, hence $t^\gamma_i=\alpha_i\odot x^\gamma_i\oplus y^\gamma_i$.

If $x^\gamma<y^\gamma=t^\gamma$ we put $y^\gamma_i=t^\gamma_i$ and
$$
x^\gamma_i=\begin{cases}
x^\gamma,&\alpha_i\odot x^\gamma<t^\gamma_i,\\
-\alpha_i\odot t^\gamma_i,&\alpha_i\odot x^\gamma\ge t^\gamma_i.\end{cases}
$$

Evidently $y^\gamma_i\rightarrow y^\gamma$.  Since $t^\gamma_i\rightarrow t^\gamma>\alpha\odot x^\gamma\leftarrow\alpha_i\odot x^\gamma$, we have $x^\gamma_i\rightarrow x^\gamma$. As before, we have $x\in[a,b]$. Finally, we have $\alpha_i\odot x^\gamma_i\le t^\gamma_i$, hence $t^\gamma_i=\alpha_i\odot x^\gamma_i\oplus y^\gamma_i$.

If $y^\gamma=\alpha\odot x^\gamma=t^\gamma$ we put $y^\gamma_i=t^\gamma_i$ and $x^\gamma_i=-\alpha_i\odot t^\gamma_i$. It is easy to see that $(x^\gamma_i,y^\gamma_i,\alpha_i,0)$ is a sequence we are looking for.

Consider the case when the set $H$ is infinite. We can assume that the net $(t^\gamma_i)_{i\in T}$ is indexed by the up-directed family of finite subsets of $H$. So, $T=\{A\subset H\mid A$ is finite $\}$. We also assume that  $\max\{t^\beta_A-t^\beta\mid\beta\in A\}<-\alpha$ for each $A\in T$. Put $\alpha_A=\alpha\odot\max\{t^\beta_A-t^\beta\mid\beta\in A\}$ for $A\in T$. We have that the net  $(\alpha_A)$ converges to $\alpha$ and $\alpha_A\le 0$ for each $A\in T$.

If $y^\gamma<\alpha\odot x^\gamma=t^\gamma$ we put
$$
x^\gamma_A=\begin{cases}
\min\{x^\gamma,-\alpha_A\odot t^\gamma_A\},&\gamma\notin A,\\
-\alpha_A\odot t^\gamma_A,&\gamma\in A\end{cases}
$$ and
$$
y^\gamma_A=\begin{cases}
y^\gamma,&y^\gamma<t^\gamma_A {\text\ and\ } \gamma\in A,\\
t^\gamma_A,&y^\gamma\ge t^\gamma_A {\text\ or\ } \gamma\notin A.\end{cases}
$$

If $x^\gamma<y^\gamma=t^\gamma$ we put $y^\gamma_A=t^\gamma_A$ and
$$
x^\gamma_A=\begin{cases}
x^\gamma,&\alpha_A\odot x^\gamma<t^\gamma_A{\text\ and\ } \gamma\in A,\\
t^\gamma_A,&\alpha_A\odot x^\gamma\ge t^\gamma_A {\text\ or\ } \gamma\notin A.\end{cases}
$$

If $y^\gamma=\alpha\odot x^\gamma=t^\gamma$ we put $y^\gamma_i=t^\gamma_i$ and
$$
x^\gamma_A=\begin{cases}
t^\gamma_A,&\gamma\notin A,\\
-\alpha_A\odot t^\gamma_A,&\gamma\in A.\end{cases}
$$ We obtain by routine checking that  $(x^\gamma_A,y^\gamma_A,\alpha_A,0)$ is a net we are looking for.
\end{proof}

The following example shows that we can not generalize the previous theorem to a product of any  $I$-barycentrically open compacta.
We consider  the two-point set $\{0,1\}$ with the discrete topology. The max-plus convex compactum $I(\{0,1\})=\{\oplus_{i=0}^1\lambda_i\odot\delta_i\mid\lambda_i\in\R_{\max},$ such that $\oplus_{i=0}^1\lambda_i=0\}$ is $I$-barycentrically open by  Corollary \ref{free}.

\begin{proposition}\label{prod} The max-plus convex compactum $I(\{0,1\})\times I(\{0,1\})$ is not  $I$-barycentrically open.
\end{proposition}

\begin{proof} We have $(\delta_0,\delta_0)\oplus(\delta_1,\delta_1)=(\delta_0\oplus\delta_1,\delta_0\oplus\delta_1)$. Define functions $\varphi_0$, $\varphi_1:\{0,1\}\to\R$ as follows $\varphi_0(0)=0$, $\varphi_0(1)=1$ and $\varphi_1(0)=1$, $\varphi_1(1)=0$. Define neighborhoods $O_0$ and $O_1$ of $(\delta_0,\delta_0)$ and $(\delta_1,\delta_1)$ respectively as follows $O_0=\{(\nu,\mu)\in I(\{0,1\})\times I(\{0,1\})\mid \nu(\varphi_0)<\frac{1}{2}$ and $\mu(\varphi_0)<\frac{1}{2}\}$ and $O_1=\{(\eta,\alpha)\in I(\{0,1\})\times I(\{0,1\})\mid \eta(\varphi_1)<\frac{1}{2}$ and $\alpha(\varphi_1)<\frac{1}{2}\}$.

The sequence $(-\frac{1}{n}\odot\delta_0\oplus\delta_1,-\frac{1}{n}\odot\delta_1\oplus\delta_0)$ converges to $(\delta_0\oplus\delta_1,\delta_0\oplus\delta_1$.
For $n\in\N$ consider $(\nu_n,\mu_n)\in I(\{0,1\})\times I(\{0,1\})$, $(\eta_n,\alpha_n)\in I(\{0,1\})\times I(\{0,1\})$ and $k_n$, $l_n\in [-\infty,0]$ such that $k_n\odot(\nu_n,\mu_n)\oplus l_n\odot(\eta_n,\alpha_n)=(-\frac{1}{n}\odot\delta_0\oplus\delta_1,-\frac{1}{n}\odot\delta_1\oplus\delta_0)$.

Consider the case $l_n=0$. Then $k_n\odot\nu_n\oplus \eta_n=-\frac{1}{n}\odot\delta_0\oplus\delta_1$. Then we have $\eta_n=e_n^0\odot\delta_0\oplus e_n^1\delta_1$ with $e_n^1\le-\frac{1}{n}$, hence  $e_n^0=0$. Then we have $\eta_n(\varphi_1)\ge 1$ and $(\eta_n,\alpha_n)\notin O_1$. In the case $k_n=0$ using analogous arguments we obtain $(\nu_n,\mu_n)\notin O_0$.

Hence the map $s_{I(\{0,1\})\times I(\{0,1\})}$ is not open. Then $I(\{0,1\})\times I(\{0,1\})$ is not  $I$-barycentrically open by Theorem \ref{main}.
\end{proof}

Hence the answer to Question 7.3 \cite{Zar} generally is negative.



A map $f:X\to Y$ between max-plus convex compacta $X$ and $Y$ is called max-plus affine if for each  $a, b\in X$ and $\alpha\in[-\infty,0]$ we have $f(\alpha\odot a\oplus  b)=\alpha\odot f(a)\oplus  f(b)$. A max-plus convex compactum $X$ is called a max-plus affine retract of a max-plus convex compactum $Y$ if there exist affine maps $r:Y\to X$ and $i:X\to Y$ such that $r\circ i=\id_X$. The map $r$ is called a retraction.

The proofs of the following two theorems are analogous to the proofs of its counterparts for probability measures (Theorems 7.5 and 7.6 from \cite{Fed2}).

\begin{theorem}\label{ar} Affine retract of an $I$-barycentrically open compactum is $I$-barycentrically open. \end{theorem}

\begin{theorem}\label{oi} Open affine image of an $I$-barycentrically open compactum is $I$-barycentrically open. \end{theorem}

Let $X$ be a max-plus convex compactum. Following \cite{GK}  we call a point $x\in X$ an extremal point if for each two points $y,z\in X$ and for each $t\in [-\infty,0]$ the equality $x=t\odot y\oplus z$  implies $x\in\{y,z\}$. The set of extremal points of a max-plus convex compactum $X$ we denote by $\ext(X)$.

\begin{theorem}\label{extr} Let $X$ be a max-plus convex compactum such that the map $\beta_X$ is open. Then the set $\ext(X)$ is closed in $X$. \end{theorem}

\begin{proof} Suppose the contrary. There exists a net $\{x_\alpha\}$ in $\ext(X)$ converging to a point $x\notin\ext(X)$.   Then there exist $y,z\in X$ and  $t\in [-\infty,0]$ such that $x=t\odot y\oplus z$ and $x\notin\{y,z\}$. Evidently, $y\neq z$. There exist open neighborhoods $V$, $U$ of $y$ and $z$ correspondingly  such that $x\notin\Cl(V\cup U)$. We can suppose that $x_\alpha\notin(V\cup U)$ for each $\alpha$. Since the map $s_X$ is open, there exist $\alpha$, $y_\alpha\in V$, $z_\alpha\in U$ and  $(p,q)\in J$ such that $x_\alpha=p\odot y_\alpha\oplus q\odot z_\alpha$. Since $x_\alpha\in\ext(X)$, we have $x_\alpha\in\{y_\alpha,z_\alpha\}\subset V\cup U$ and we obtain a contradiction.
\end{proof}


An example of a convex compactum with the closed set of extremal points and not open barycenter map was build in \cite{OBr}. We construct an idempotent counterpart. Consider a subset $Y\subset [-2,0]^2$ defined as follows $Y=A\cup B\cup C$ where $A=\{(x,y)\in [-2,0]^2\mid x\in [-2,-1],\ y=-1\}$, $B=\{(x,y)\in [-2,0]^2\mid x=-1,\ y\in [-2,-1]\}$ and $C=\{(x,y)\in [-1,0]^2\mid x=y\}$. It is easy to see that $Y$ is a max-plus convex compactum. Consider points $a=(-2,-1)$, $b=(-1,-2)$, $c=(-1,-1)$  and  a sequence $(c_i)$ where $c_i=(-1+\frac{1}{i}, -1+\frac{1}{i})$ for $i\in\N$. Evidently the sequence $(c_i)$ converges to $c$. Put $\nu=\delta_a\oplus \delta_b$. We have $\beta_Y(\nu)=c$. It is easy to check that there is no sequence $(\nu_i)$ converging to $\nu$ and such that $b_Y(\nu_i)=c_i$. Hence $\beta_Y$ is not open  but  $\ext(X)=\{a,b,(0,0)\}$.

\end{document}